\documentclass[a4paper,11pt]{amsart}
\usepackage{amsmath,amssymb,mathrsfs,mathtools,microtype}
\usepackage[all,cmtip]{xy}
\usepackage[margin=1in]{geometry}

\usepackage[colorlinks,allcolors=blue,final,backref=page,linktocpage]{hyperref}

\newtheorem{thm}{Theorem}[section]
\newtheorem{cor}[thm]{Corollary}

\newtheorem{prop}[thm]{Proposition}
\theoremstyle{definition}
\newtheorem{defi}[thm]{Definition}
\newtheorem{example}[thm]{Example}
\newtheorem{rem}[thm]{Remark}

\numberwithin{equation}{section}

\newcommand{\op}{\mathrm{op}}
\def\C{\mathcal{C}}
\def\As{\operatorname{Ass}}
\def\W{\mathbf{W}}
\def\Cof{\mathbf{C}}
\def\F{\mathbf{F}}
\def\Int{\mathcal{I}}
\def\ho{\operatorname{ho}}
\def\Sset{\operatorname{SSet}}
\def\MCmod{\mathscr{MC}}

\def\Map{\operatorname{Map}}
\def\CDGA{\operatorname{CDGA}}
\def\DGLA{\operatorname{DGLA}}
\def\DGA{\operatorname{DGA}}
\def\DGAa{\operatorname{DGA/k}}
\def\pcVect{\operatorname{pcVect}}
\def\DGVect{\operatorname{DGVect}}
\def\pcDGVect{\operatorname{pcDGVect}}
\def\pcDGA{\operatorname{pcDGA}}
\def\pcCDGA{\operatorname{pcCDGA}}
\def\pcDGAloc{\operatorname{pcDGA_{loc}}}
\def\pcDGAlocop{\operatorname{pcDGA_{loc}^{\op}}}
\def\pcCDGAloc{\operatorname{pcCDGA_{loc}}}
\def\pcCDGAlocop{\operatorname{pcCDGA_{loc}^{op}}}

\newcommand{\Lie}{\mathsf{Lie}}

\newcommand{\sll}{\mathfrak{sl}}
\newcommand{\Vect}{\mathrm{Vect}}

\def\Bar{\operatorname{Bar}}
\def\Ab{\operatorname{Ab}}
\def\Cobar{\operatorname{Cobar}}

\def\Harr{\operatorname{Harr}}
\def\CE{\operatorname{CE}}
\DeclareMathOperator{\MC}{MC}

\DeclareMathOperator{\Def}{Def}

\DeclareMathOperator{\REnd}{REnd}
\DeclareMathOperator{\Hom}{Hom}
\DeclareMathOperator{\Ker}{Ker}

\def\id{\operatorname{id}}

\newcommand{\nc}{\newcommand}
\newcommand{\delete}[1]{}
\nc{\mlabel}[1]{\label{#1}}  
\nc{\mcite}[1]{\cite{#1}}  
\nc{\mref}[1]{\ref{#1}}  
\nc{\mbibitem}[1]{\bibitem{#1}} 
\delete{
\nc{\mlabel}[1]{\label{#1}{\hfill \hspace{1cm}{\bf{{\ }\hfill(#1)}}}}
\nc{\mcite}[1]{\cite{#1}{{\bf{{\ }(#1)}}}}  
\nc{\mref}[1]{\ref{#1}{{\bf{{\ }(#1)}}}}  
\nc{\mbibitem}[1]{\bibitem[\bf #1]{#1}} 
}
\newcommand{\emptycomment}[1]{}

\delete{
\setlength{\baselineskip}{1.8\baselineskip}

\newcommand{\emptycomment}[1]{}

}

\newcommand{\g}{\mathfrak g}
\newcommand{\h}{\mathfrak h}
\newcommand{\I}{\mathfrak I}

\begin{document}

\title[Review of deformation theory II]{Review of deformation theory II:\\ a homotopical approach}

\author{Ai Guan}
\address{Department of Mathematics and Statistics, Lancaster University, Lancaster LA1 4YF, UK}
\email{a.guan@lancaster.ac.uk}

\author{Andrey Lazarev}
\address{Department of Mathematics and Statistics, Lancaster University, Lancaster LA1 4YF, UK}
\email{a.lazarev@lancaster.ac.uk}

\author{Yunhe Sheng}
\address{Department of Mathematics, Jilin University, Changchun 130012, Jilin, China}
\email{shengyh@jlu.edu.cn}

\author{Rong Tang}
\address{Department of Mathematics, Jilin University, Changchun 130012, Jilin, China}
\email{tangrong16@mails.jlu.edu.cn}

\date{\today}

\begin{abstract}
We give a general treatment of deformation theory from the point of view of homotopical algebra following Hinich, Manetti and Pridham. In particular, we show that any deformation functor in characteristic zero is controlled by a certain differential graded Lie algebra defined up to homotopy, and also formulate a noncommutative analogue of this result valid in any characteristic.
\end{abstract}

\subjclass[2010]{13D03,13D10,17B56,17A30}

\keywords{model category, deformation, Maurer-Cartan element, associative algebra, Lie algebra}

\maketitle
\tableofcontents
\allowdisplaybreaks

\section{Introduction}\mlabel{sec:intr}

It has long been observed that any reasonable deformation theory is `controlled' by a differential graded (dg) Lie algebra, at least in characteristic zero. The goal of this paper is to explain how this slogan can be made into a rigorous theorem. The work \cite{Pridham} contains essentially the same result but our approach, while conceptually close to op.~cit., is more elementary, e.g.~it does not use simplicial techniques in any essential way. The work by Lurie \cite{Lurie} also contains a version of this result, formulated in the framework of infinity categories.

The approach to abstract deformation theory that we adopt rests on two fundamental results that are important and interesting in their own right: Koszul duality between dg Lie algebras and cocommutative dg conilpotent coalgebras as formulated by Hinich \cite{Hinich} (as well as the associative variant \cite{Positselski}) and an abstract version of Brown's representability theorem \cite{Heller, Jardine}. We will be using the language of Quillen closed model categories \cite{Hovey,Quillen}.

The standard approach to deformation theory is as follows. Suppose that $O$ is an object of a certain category that one wants to deform and which is defined, in some sense, over some ground field $k$. We will not attempt to axiomatize this situation but a good example to keep in mind is an associative algebra over $k$. Then, a deformation of $O$ over a finite dimensional (Artinian) local ring $K$ is an object $O_K$ defined, in the same vague sense, over $K$ and such that its `reduction' modulo the maximal ideal $I$ of $K$ is isomorphic to $O$. Two deformations $O_K$ and $O^\prime_K$ are \emph{equivalent} if $O_K$ and $O_K^\prime$ are isomorphic via an isomorphism that is the identity modulo $I$. Thus, we have a functor $\Def$ associating to a local Artinian ring $K$ the set of equivalence classes of deformations of $O$ over $K$. The fundamental problem of deformation theory is finding a `universal' ring $K_u$ and the corresponding universal deformation of $O$ over $K_u$, i.e.~an element in $\Def(K_u)$ so that any other deformation of $O$ over $K$
is induced by a unique map $K_u\to K$.

It has been understood for a long time that one can only expect the deformation functor to be \emph{pro-representable}, in other words
we might as well extend the category on which $\Def$ is defined to include projective limits of local Artinian $k$-algebras (we will call them local pseudocompact $k$-algebras). Furthermore, it makes sense to attempt to characterize functors on the category of local Artinian (or pseudocompact) rings that deformation functors satisfy in concrete examples and then investigate whether these characteristic properties ensure (pro)representability. Note that a representable functor preserves arbitrary limits; moreover under some mild conditions on the category of set-theoretical nature, any functor preserving limits is representable (the so-called Freyd's adjoint functor theorem, \cite{Freyd}). However, a deformation functor may not preserve limits; indeed infinitesimal automorphisms often present an obstruction to such a preservation, cf.~\cite[Remark 2.15]{Schlessinger} for an explanation of this point. On the other hand, we often have that it preserves arbitrary products and that the natural map of sets
\begin{equation}\label{fibreproduct}
\Def(B\times_A C)\to \Def(B)\times_{\Def(A)}\Def(C)
\end{equation}
is \emph{surjective} (if it is bijective this would imply that $\Def$ preserves arbitrary colimits). Additionally, it usually makes sense to impose the \emph{normalization condition}: $\Def(k)$ is a one-point set. Together with another mild condition on infinitesimal deformations, these imply that $\Def$ has a \emph{hull}, a certain weakening of the property of being representable, \cite[Theorem 2.11]{Schlessinger}.

In order to obtain a decisive general result, it is necessary to extend the category of local pseudocompact algebras to that of local \emph{differential graded} pseudocompact algebras. The advantage of the latter is that it has the structure of a \emph{closed model category} and, in particular one can form its \emph{homotopy category}. This closed model category was constructed in a seminal paper of Hinich \cite{Hinich} and an extended deformation functor was considered in \cite{Manetti, Merkulov}. The latter papers, however, did not make full use of the strength of the closed model structure on local pseudocompact dg algebras.

So, we now have a set-valued functor defined on the category of local pseudocompact dg algebras. It is a deformation functor if it is normalized, preserves arbitrary products, has an appropriate analogue of (\ref{fibreproduct}) and, crucially, is \emph{homotopy invariant}, so that it descends to a functor on the homotopy category of local pseudocompact dg algebras. In the commutative case and when $k$ has characteristic zero, we will show that, under these  conditions the functor is representable in the homotopy category and there is a certain dg Lie algebra, defined up to a quasi-isomorphism `controlling' it. In the associative case we will similarly show that, under these conditions the functor is representable in the homotopy category and there is a certain dg \emph{associative} algebra, defined up to a quasi-isomorphism `controlling' it; this will be valid in any characteristic.

This article is the second part of a review of deformation theory, although it can be read independently. The first part \cite{GLST} gives concrete formulas for dg Lie algebras controlling deformations of various algebraic structures: associative algebras, Lie algebras, pre-Lie algebras, Leibniz algebras, 3-Lie algebras and some of their $\infty$-variants. We refer to the first part for examples of the theory presented here.


\section{Closed model categories}

Closed model categories (with the adjective `closed' frequently omitted) were introduced in \cite{Quillen}, as an abstraction of the category of topological spaces or simplicial sets. However it quickly became clear that this notion has much wider applicability, in particular much of the classical homological algebra can be formulated in the language of closed model categories. We will see that deformation theory can likewise be profitably recast in this language. The survey \cite{DWSpalinski} covers most of our needs; for more in-depth treatment see \cite{Hirschhorn,Hovey}.
\begin{defi}
A category $\C$ is a \emph{model category} if it is supplied with three classes of morphisms: weak equivalences $\mathbf W$, fibrations $\F$ and cofibrations $\Cof$, each closed under compositions and containing identity maps. Morphisms in $\Cof \cap \W$ are called \emph{acyclic cofibrations} and morphisms in $\F\cap\W$ are similarly called \emph{acyclic fibrations}. The following axioms are required to hold.
\begin{enumerate}
\item[(MC1)]
$\C$ contains arbitrary limits and colimits;
\item[(MC2)]
If $f,g$ are morphisms in $\C$ for which $f\circ g$ is defined and out of three morphisms $f, g, f\circ g$ two are weak equivalences, then so is the third.
\item[(MC3)] The classes of morphisms $\W$, $\Cof$ and $\F$ are each closed under retracts.
\item[(MC4)] Given a diagram in $\C$ of the form
\begin{equation}\label{eq:LLPRLP}
\xymatrix{A\ar_i[d]\ar[r]&X\ar^p[d]\\
B\ar[r]\ar@{-->}[ur]&Y		
}
\end{equation}
the dotted arrow making the whole diagram commutative exists if either	
\begin{enumerate}
\item $i\in\Cof\cap \W$ and $p\in\F$ or
\item $i\in\Cof$ and $p\in \F\cap \W$.
\end{enumerate}
If the dotted arrow in (\ref{eq:LLPRLP}) exists, we say that $i$ has the Left Lifting Property (LLP) with respect
to $p$ and $p$ has the Right Lifting Property (RLP) with respect to $i$.
\item[(MC5)] Any morphism $f$ in $\C$ can functorially be factored in two ways:
\begin{enumerate}
\item as $f=p\circ i$ where $p\in\F$ and $i\in\Cof\cap \W$ or
\item as $f=p\circ i$ where $p\in\F\cap\W$ and $i\in\Cof$.
\end{enumerate}
\end{enumerate}
\end{defi}
\begin{rem}
The above definition differs from the original one by Quillen in that the latter only assumes the existence of finite limits and colimits and the factorizations of maps as in Axiom MC5 were not required to be functorial. However, in practice, the strengthened axioms hold in most of the cases of interest and this modification is often preferred in the current literature.
\end{rem}
Due to existence of limits and colimits, a model category $\C$ has an initial object $\varnothing$ and a terminal object $*$; if these are isomorphic, $\C$ is called a \emph{pointed} model category. An object $X$ of $\C$ is called \emph{fibrant} if the unique map $X\to *$ is a fibration and \emph{cofibrant} if the map $\varnothing \to X$ is a cofibration. By the factorization axiom MC5, for every object $X$ there is functorially associated with it a fibrant object $RX$ and an acyclic cofibration $X\to RX$; similarly there is a cofibrant object $LX$ and an acyclic fibration $LX\to X$. We will call $RX$ and $LX$ \emph{fibrant and cofibrant replacements} of $X$, respectively. Moreover, it is easy to see that any object $X\in\C$ can be connected by (possibly a zigzag of) weak equivalences to an object that is both fibrant and cofibrant; e.g.~such is the object $L(RX)$ or $R(LX)$.

\begin{example}\label{ex:closedmodel}
Here are a few examples of  model categories.
\begin{enumerate}
\item The category $\operatorname{Top}$ of topological spaces is a closed model category where weak equivalences are the ordinary weak equivalences of topological spaces, fibrations are Serre fibrations and cofibrations are those maps that have the LLP with respect to Serre fibrations. All objects are fibrant and the cofibrant objects are retracts of CW complexes. This is the prototypical model category that served as a blueprint and motivation for developing the whole theory of model categories.
\item The category $\Sset$ of simplicial sets, \cite{GoerssJardine}. The weak equivalences are those maps between simplicial sets $S\to K$ such that the induced map on their geometric realizations $|S|\to|K|$ is a weak equivalence of topological spaces. Cofibrations are injections of simplicial sets and fibrations are the maps having the RLP with respect to acyclic cofibrations.
\item The category $\operatorname{Ch}(R)$ of chain complexes over an associative ring $R$ has two natural model category structures with weak equivalences being quasi-isomorphisms of chain complexes. In the first one (called the \emph{projective model structure}) fibrations are surjective maps and cofibrations are chain maps having the LLP with respect to surjective chain maps, whereas in the second one (called the \emph{injective model structure}) cofibrations are injective maps and fibrations are chain maps having the RLP with respect to injective chain maps. Much of classical homological algebra can be formulated in terms of these model categories.
\item The categories $\CDGA$ and $\DGLA$ of commutative dg algebras and dg Lie algebras over a field of characteristic zero and $\DGA$ and $\DGAa$ of dg algebras and \emph{augmented} dg algebras over a field of arbitrary characteristic have model structures where weak equivalences are quasi-isomorphisms, fibrations are surjective maps and cofibrations are the maps having the LLP with respect to fibrations. All objects are fibrant in these model categories.
\end{enumerate}
\end{example}

Next, we will discuss the notion of \emph{homotopy} in model categories; these are based on \emph{path and cylinder objects}.

\begin{defi}
Let $X$ be an object in a model category $\C$.
\begin{enumerate}
\item A \emph{cylinder object} for $X$ is an object $X\otimes I$ together with a factorization
$X\coprod X\xrightarrow{i}X\otimes I\xrightarrow{p} X$
of the canonical folding map $X\coprod X\to X$ into a cofibration followed by an acyclic fibration.
Two maps $f,g:X\to Y$ in $\C$ are said to be \emph{left homotopic} if their sum $f\coprod g:X\coprod X\to Y$ extends to a map $X\otimes I\to Y$.
\item A \emph{path object} for $X$ is an object $X^I$ together with a factorization
$X\xrightarrow{i}X^I\xrightarrow{p} X\times X$
of the canonical diagonal map $X\to X\times X$ into an acyclic cofibration followed by a fibration.
Two maps $f,g:X\to Y$ in $\C$ are said to be \emph{right homotopic} if the product map $(f,g):X\to Y\times Y$ lifts to a map $X\to Y^I$.
\end{enumerate}
\end{defi}

\begin{rem}
Some authors prefer to weaken the notions of a cylinder and path object, for example, not insisting that the map $X\coprod X\to X\otimes I$ be a cofibration (note that in the case of topological spaces the standard topological cylinder $X\times[0,1]$ this condition is not satisfied unless $X$ is a CW complex). Nevertheless, the axiom MC5 ensures that any object has a functorial cylinder and path object.
\end{rem}
The following result holds.

\begin{thm}\label{th:homotopyclasses}
Let $X$ be a cofibrant object and $Y$ be a fibrant object of a model category $\C$. Then \begin{enumerate}
\item Two maps $X\to Y$ are left homotopic if and only if they are right homotopic.
\item The relation of left or right homotopy on $\Hom_{\C}(X,Y)$ is an equivalence relation. The set of (left or right) homotopy classes of maps $X\to Y$ will be denoted by $[X,Y]$.
\item If $X^\prime$ is a cofibrant object weakly equivalent to $X$ and $Y^\prime$ is a fibrant object weakly equivalent to $Y$ then there is a bijection $[X,Y]\cong [X^\prime, Y^\prime]$.
\item If $f,g:X\to Y$ are left homotopic and $h:A\to X$ is a map with $A$ cofibrant, then $h\circ f$ and $h\circ g$ are left homotopic. Similarly if $k:Y\to B$ is a map with $B$ fibrant then $f\circ k$ and $g\circ k$ are right homotopic.
\item Suppose additionally that $X,Y\in\C$ are both fibrant and cofibrant and that $f:X\to Y$ is a weak equivalence. Then
$X$ and $Y$ are \emph{homotopy equivalent}, i.e.~there exists a map $g:Y\to X$ such that $f\circ g$ is homotopic to $\id_Y$ and $g\circ f$ is homotopic to $\id_X$.
\end{enumerate}
\end{thm}

\begin{proof}
See \cite[Section 4]{DWSpalinski}.
\end{proof}	

This allows one to construct the homotopy category of a model category.

\begin{defi}
The \emph{homotopy category} of a model category $\C$ is the category $\ho\C$ whose objects are the objects in $\C$ that are both fibrant and cofibrant and for two fibrant-cofibrant objects $X, Y\in \C$ we have $\Hom_{\ho\C}(X,Y):=[X,Y]$, the homotopy classes of maps from $X$ to $Y$.	
\end{defi}
Theorem~\ref{th:homotopyclasses} ensures that $\ho\C$ is well-defined. Moreover, the correspondence $X\mapsto L(RX)$ (or, equivalently, $X\mapsto R(LX)$) determines a functor $\gamma:\C\to\ho\C$. It follows from Theorem~\ref{th:homotopyclasses} (5) that $\gamma$ takes weak equivalences in $\C$ into isomorphisms in $\ho\C$; it is remarkable that $\gamma$ is the \emph{universal functor} out of $\C$ having this property.

\begin{thm}
Let $F:\C\to\mathcal{D}$ be a functor from a model category $\C$ to a category $\mathcal{D}$ such that for any weak equivalence $f\in\C$ its image $F(f)\in\mathcal{D}$ is an isomorphism. Then there exists a unique functor $G:\ho\C\to\mathcal{D}$ such that $G\circ\gamma=F$.
\end{thm}

\begin{proof}
See \cite[Theorem 4.2]{DWSpalinski}.
\end{proof}
\begin{rem}
The homotopy category of a  model category $\C$ is where the most important invariants of $\C$ lie. For example, the \emph{derived category} of a ring $R$ is the homotopy category of $\operatorname{Ch}(R)$, cf.~Example~\ref{ex:closedmodel} with either projective or injective model structure. Thus, different model structures on the same category may lead to equivalent homotopy categories.
\end{rem}
Having defined the notion of a model category, it is natural to consider functors between different model categories. It is unreasonable to require that functors preserve the whole structure available (i.e.~all classes $\W$, $\F$, $\Cof$) as this does not hold in many cases of interest. The appropriate notion here is that of a \emph{Quillen adjunction}.
\begin{defi}
Let $\C$ and $\mathcal D$ be model categories and $F:\C\rightleftarrows \mathcal{D}:G$ be an adjoint pair of functors so that $F$ is left adjoint to $G$.
We say that $(F,G)$ is a Quillen adjunction if $F$ takes cofibrations in $\C$ to cofibrations in $\mathcal {D}$ and $G$ takes fibrations in $\mathcal D$ to fibrations in $\C$. We will refer to $F$ as a \emph{left Quillen functor} and to $G$ as a \emph{right Quillen functor}.
\end{defi}
If $F:\C\rightleftarrows \mathcal{D}:G$ is a Quillen adjunction, then one can prove that $F$ carries weak equivalences between cofibrant object into weak equivalences and likewise $G$ carries weak equivalences between fibrant objects into weak equivalences. It follows that $F$ and $G$ lift to functors $LF$ and $RG$ between the corresponding homotopy categories $\ho\C$ and $\ho\mathcal D$.
We will refer to $LF$ as the \emph{left derived functor} of $F$ and to $RG$ as the \emph{right derived functor} of $G$.
Moreover, $(LF, RG)$ also form an adjoint pair:
\begin{thm}
A Quillen adjunction $F:\C\rightleftarrows\mathcal{D}:G$ determines an (ordinary) adjunction
\[
LF:\ho\C\rightleftarrows\ho\mathcal{D}:RG.
\]
\end{thm}
\begin{proof} See \cite[Theorem 9.7]{DWSpalinski}.
\end{proof}
\begin{defi}
A Quillen adjunction $F:\C\rightleftarrows \mathcal{D}:G$ is called a \emph{Quillen equivalence} if the corresponding adjunction $ LF:\ho\C\rightleftarrows \ho\mathcal{D} : RG$ is an ordinary equivalence.
\end{defi}
\begin{example}\
\begin{itemize}
\item Let $R$ be an associative ring and $\C$ be the category of chain complexes of $R$ modules with its projective model structure, and $\mathcal D$ be the same category with the injective model structure, cf.~Example \ref{ex:closedmodel}(3). Then the identity functor $\C\to \mathcal D$ is a right Quillen functor establishing
a Quillen equivalence between $\C$ and $\mathcal D$. Its adjoint left Quillen functor $\mathcal D\to\C$ is, of course, also the identity functor. Informally, this can be interpreted as saying that there are two equivalent approaches to classical homological functors: one based on injective resolutions and the other based on projective resolutions.
\item
The functor of geometric realization from simplicial sets to topological spaces is a left Quillen functor whose right adjoint is the functor associating to a topological space its singular simplicial set \cite{Hovey}. This adjunction is a Quillen equivalence.

\item Later on we will consider Koszul duality as a Quillen equivalence between the categories of commutative pseudocompact dg algebras and dg Lie algebras and see that it underlies the modern approach to deformation theory.
\end{itemize}
\end{example}

In a  model category $\C$ one can define the notions of homotopy pullbacks and homotopy pushouts; an elementary construction can be found in \cite[Section 10]{DWSpalinski}.
\begin{defi}
Let $X,Y$ and $Z$ be objects in a closed model category $\C$ supplied with maps $X\to Y$ and $X\to Z$. Factor the map $LX\to X\to Y$ as $LX\xrightarrow{i_1}\tilde{Y}\xrightarrow{p_1} Y$ where $i_1$ is a cofibration and $p_1$ is an acyclic fibration; similarly factor the map $LX\to X\to Z$ as $LX\xrightarrow{i_2}\tilde{Z}\xrightarrow{p_2} Z$ where $i_2$ is a cofibration and $p_2$ is an acyclic fibration.
Then the \emph{homotopy pushout} $Y\coprod^h_XZ$ is by definition $ \tilde{Y}\coprod_{LX}\tilde{Z}$.

A \emph{homotopy pullback} is defined dually as a homotopy pushout in $\C^{op}$. It will be denoted for objects $X,Y$ and $Z$ by $Y\times^h_XZ$.
\end{defi}
\begin{rem}
The notions of a homotopy pullbacks and pushout are derived functors of ordinary pullbacks and pushouts. Namely, consider the category of diagrams $\operatorname{Push}(\C)$ in a model category $\C$ of the form $Y\leftarrow X\rightarrow Z$ and a functor $F:\operatorname{Push}(\C)\to\C$ obtained by taking the pushout of a given diagram. Then there exists a model structure on $\operatorname{Push}(\C)$ such that $F$ is a left Quillen functor and then the homotopy pushout is its left derived functor. The case of a homotopy pullback is similar.
\end{rem}

Homotopy pushouts and pullbacks are simplified in \emph{proper} model categories.
\begin{defi}
A model category $\C$ is called \emph{left proper} if for any pushout diagram in $\C$
\[\xymatrix{
A\ar^i[r]\ar_f[d]&B\ar^g[d]\\
C\ar[r]&D	
}\]
for which $i$ is a cofibration and $f$ is a weak equivalence, then the map $g$ is also a weak equivalence.
Dually, $\C$ is \emph{right proper} if for any pullback diagram in $\C$
\[\xymatrix{
A\ar[r]\ar_f[d]&B\ar^g[d]\\
C\ar_p[r]&D	
}\]
for which $p$ is a fibration and $g$ is a weak equivalence, then the map $f$ is also a weak equivalence.
\end{defi}
Many model categories are left or right proper as the following result makes clear.
\begin{prop}\label{prop:rightproper}
Let $\C$ be a model category such that every object of $\C$ is cofibrant. Then $\C$ is left proper. Dually, if every object of $\C$ is fibrant, then $\C$ is right proper.
\end{prop}
\begin{proof}
See \cite[Proposition A.2.4.2]{Lurie}.
\end{proof}	
Then the following result holds.
\begin{prop}
Let $Y\leftarrow X\rightarrow Z$ be a diagram in a left proper model category where $X\to Y$ is a cofibration. Then
$Y\coprod_X Z$ is weakly equivalent to $Y\coprod^h_X Z$.

Dually, let $Y\rightarrow X\leftarrow Z$ be a diagram in a right proper model category where $Z\to X$ is a fibration.
Then $Y\times_XZ$ is weakly equivalent to $Y\times^h_XZ$.
\end{prop}
\begin{proof}
See \cite[Proposition A.2.4.4]{Lurie}.
\end{proof}	
Lastly, we discuss the existence of derived mapping spaces in model categories.
\begin{thm}Let $X$ be a cofibrant object and $Y$ be a fibrant object in a model category $\C$.
\begin{enumerate}
\item For any object $A$ there exists a simplicial set $\Map_l(A,Y)$, such that $\pi_0\Map_l(A,Y)\cong[A,Y]_l$ a simplicial set $\Map_r(X,A)$ such that $\pi_0\Map_r(X,A)\cong[X,A]_r$. \item The functors $A\mapsto\Map_l(A,Y)$ and $A\mapsto\Map_r(X,A)$ are left and right Quillen functors from $\C$ to simplicial sets respectively.

\item There is a natural isomorphism $\Map_l(X,Y)\cong\Map_r(X,Y)$.
\end{enumerate}
\end{thm}
\begin{proof}
See \cite[Section 5.4]{Hovey}.	
\end{proof}
When $X$ is cofibrant and $Y$ is fibrant, we will write $\Map(X,Y)$ for either $\Map_l(X,Y)$ or $\Map_r(X,Y)$ and call it the \emph{derived mapping space} from $X$ to $Y$.		

\section{Brown representability theorem for compactly generated model categories}

The Brown representability theorem \cite{Brown} is a necessary and sufficient condition for a functor defined on the homotopy category of pointed topological spaces to be representable. It has subsequently been formulated in various abstract contexts. It will be convenient for us to use a version due to Jardine, \cite{Jardine}.
\begin{defi}
Let $\C$ be a closed model category. We say that $\C$ is \emph{compactly generated} if there exists a set $S$ of compact
cofibrant objects in $\C$ that \emph{detect weak equivalences}, i.e.~a map $X\to Y$ in $\C$ is a weak equivalence if and only if for any $K\in S$ there is a bijection $[K,X]\to [K,Y]$.
\end{defi}
\begin{example}
The category of connected pointed topological spaces is compactly generated with $S:=S^n, n=1,2,\ldots$, the pointed spheres. It is interesting to note that the category of all (i.e.~not necessarily connected) topological spaces is not compactly generated, \cite{Heller}.
\end{example}
\begin{rem}
There is another, inequivalent notion of a compactly generated closed model category contained in e.g.~\cite{MayPonto}. Under this notion the category of all topological spaces \emph{is} compactly generated.
\end{rem}
Under the assumption of compact generation, an abstract Brown representability holds in $\C$.

\begin{thm}\label{thm:brownrep}
Let $\C$ be a compactly generated pointed closed model category with $*$ denoting its initial-terminal object. Suppose that a set valued contravariant functor $F$ on $\C$
satisfies the following conditions:
\begin{enumerate}
\item $F(*)=*$,
\item $F$ takes weak equivalences to bijections of sets,
\item $F$ takes arbitrary coproducts of cofibrant objects in $\C$ into products of sets.
\item Let $A$, $B$, $C$ be cofibrant objects in $\C$ and $A\to B$, $A\to C$ be morphisms in $\C$ with $A\to B$ being a cofibration.
Then the natural map $F(B\coprod_AC)\to F(B)\times_{F(A)}F(C)$ is a surjection of sets.
\end{enumerate}
Then the functor $F$ is representable in the homotopy category of $\C$, i.e.~there exists an object $X$ in $\C$ and a natural weak equivalence $F(Y)\simeq [Y,X]$ for any $Y\in\C$.
\end{thm}

\begin{proof}
This is \cite[Theorem 19]{Jardine}.
\end{proof}

\begin{rem} Theorem \ref{thm:brownrep} is a model category version of the famous Brown representability theorem \cite{Brown} that was originally formulated in the category of pointed CW complexes. It is not the most general form of Brown's representability theorem (for such a statement see \cite{Heller}) since it can be formulated in a way not requiring the existence of a closed model structure. In practice (and particularly for the application we have in mind) a  model structure is often present and the conditions of the theorem are usually not difficult to verify.
\end{rem}
\begin{rem}
It is easy to see that, conversely, a representable up to homotopy set-valued functor on a compactly generated model category must satisfy the conditions listed in Theorem \ref{thm:brownrep}. To make a comparison with topology easier, we will view $F$ as a \emph{covariant} functor on $\C^{\op}$ represented by $X\in \C^{\op}$; we will assume without loss of generality that $X$ is cofibrant. Thus, for $Y\in \C^{\op}$ we have $F(Y)=[X,Y]$. The conditions (1), (2) and (3) are obvious. Applying $\Map(X,-)$ to a homotopy pullback of $B\to A\leftarrow C$ in $\C$, we obtain a homotopy pullback of simplicial sets (since $\Map(X, -)$ is a right Quillen functor).
\[
\xymatrix{
\Map(X,B\times^h_A C)\ar[r]\ar[d]&\Map(X,B)\ar[d]\\
\Map(X,C)\ar[r]&\Map(X,A)	
}
\]
Taking the connected components functor, we obtain a surjection \[F(B\times^h_A C)\cong\pi_0\Map(X,B\times^h_A C)\to
\pi_0\Map(X,B)\times_{\pi_0\Map(X,A)}\pi_0\Map(X,B)\cong F(B)\times_{F(A)}F(C)\]
as required.

Note also that this argument shows that one should not, in general, expect that the map	$F(B\times^h_A C)\to F(B)\times_{F(A)}F(B)$ is an isomorphism. Indeed, it follows from the homotopy pullback diagram above that the homotopy fibre
of the map \[
\Map(X,B\times^h_AC)\to\Map(X,B) \times\Map(X,C)
\] over a given point $(f,g)\in \Map(X,B) \times\Map(X,C)$ having the same image in $[X,A]$ is the based loop space $\Omega\Map(X,A)$. Thus, the fibration \[\Omega\Map(X,a)\to\Map(X,B\times^h_AC)\to\Map(X,B) \times\Map(X,C)\] gives rise to a long homotopy exact sequence (the Mayer-Vietoris sequence, \cite{Droitberg}).
\[\ldots\to\pi_1\Map(X,B)\times \pi_1\Map(X,C)\to\pi_1\Map(X,A)\to F(B\times^h_A C)\to F(B)\times_{F(A)}F(B).
\]
\end{rem}

\section{MC elements and MC moduli sets}

We will outline here the general theory of Maurer-Cartan (MC) elements in dg Lie and associative algebras and related moduli sets.
\subsection{MC moduli in dg Lie algebras}
\begin{defi}Let $\g$ be a dg Lie algebra over a field $k$ of characteristic zero. An element $x\in \g^1$ is called an \emph{MC element} if it satisfies the following
equation (called the MC or master equation)
\[
d(x)+\frac{1}{2}[x,x]=0.
\]
The set of MC elements in $\g$ will be denoted by $\MC(\g)$. If $A$ is a commutative dg algebra then $\g\otimes A$ has
naturally the structure of a dg Lie algebra and we will write $\MC(\g,A)$ for $\MC(\g\otimes A)$.

\end{defi}

From now on we shall assume that $\g$ is nilpotent or, more generally, pro-nilpotent (i.e. $\g\cong\varprojlim_n\g/\g^{[n]}$ where $\g^{[n]}$ is the dg Lie ideal generated by Lie products of at least $n$ elements). In this case it has a group $G$ associated to it.
To define $G$, recall that $U\g$, the universal enveloping algebra of $\g$ is the graded associative algebra
obtained by quotienting out the tensor algebra $T\g$ by the ideal generated by the relations $a\otimes b-(-1)^{|a||b|}b\otimes a-[a,b]$ for two homogeneous elements $a,b\in \g$.
By definition there is a map $\g \to U\g$ that turns out to be an embedding. The algebra $U\g$ is a bialgebra with the elements of $\g\subset U\g$ being primitive elements; moreover the set of primitive elements in $U\g$ coincides with $\g$. There is also an augmentation $U\g\to k$ that sends all elements of $\g$ to zero.

We will need to consider the completion $\hat{U}\g$ of $U\g$ at its augmentation ideal $\I$; i.e.~$\hat{U}\g\cong \varprojlim_n U\g/\I^n$. Note that for a general dg Lie algebra $\g$ it may happen that $\hat{U}\g=0$, such is the case, e.g.~when $\g$ is an ordinary semisimple Lie algebra. However when $\g$ is pro-nilpotent, $\hat{U}\g$ is always nontrivial; moreover the natural map $U\g\to\hat{U}\g$ is an embedding and so, $\g$ is likewise a subspace of $\hat{U}\g$. Then we define the group $G$ as the group of group-like elements in $\hat{U}\g$, i.e.~the set of elements $g\in\hat{U}\g$ such that $\Delta(g)=g\otimes g$.

There is, in fact, an equivalence of categories between pro-nilpotent Lie algebras, pro-nilpotent Lie groups and complete cocommutative Hopf algebras, cf.~\cite[Appendix A3]{Quillen}.

The group $G$ is called the \emph{gauge group} and acts on $\MC(\g)$ by gauge transformations:
\begin{prop}\label{prop:gaugeaction}
Let $g\in G$ and $x\in\MC(\g)$. Both elements $g$ and $x$ are viewed as lying in $\hat{U}\g$. Then the formula $g\cdot x:=gxg^{-1}-d(g)g^{-1}$ determines an action of $G$ on $\MC(\g)$.
\end{prop}
\begin{proof}
First note that if $\g$ has vanishing differential then the MC condition takes the form $[x,x]=0$ and the gauge action reduces to ordinary conjugation; the desired statement in this case is clear. We will reduce the general case to this one as follows. Introduce the graded Lie algebra $\tilde{\g}$ having underlying graded vector space $\g\oplus k\cdot d$ where $k\cdot d$ is the one-dimensional Lie algebra spanned by a symbol $d$ sitting in cohomological degree 1. By definition for $a\in\tilde{\g}$ we have $[d,a] := d(a)$, $[d,d]=0$ whereas $\g$ is a Lie subalgebra in $\tilde{\g}$.
Given $y\in\g$ denote by $\tilde{y}$ the element $y+d\in\tilde{\g}$. A straightforward check shows that an odd element
$x\in\g$ is Maurer-Cartan if and only if $[\tilde{x},\tilde{x}]= 0$. We will view an element
$g\in\g$ as an element in $\hat{U}\tilde{\g}$ via the embedding $\g\subset\tilde{\g}\subset\hat{U}\tilde{\g}$.
Since $d(g)=[d,g]=dg-gd\in \hat{U}\tilde{g}$ we have $d(g)g^{-1}=d-gdg^{-1}$ and so
\begin{align*}g\tilde{x}g^{-1} &= g(x+d)g^{-1}\\
&= gxg^{-1}+gdg^{-1}\\
&= gxg^{-1}+d-d(g)g^{-1}\\
&= \widetilde{g\cdot x}.
\end{align*}
So, any MC element $x\in \g$ gives rise to an MC element $\tilde{x}\in\tilde{\g}$ where $\tilde{\g}$ has vanishing differential and the gauge action in $\g$ corresponds to the conjugation action in $\tilde{\g}$. The desired statement is now obvious.
\end{proof}

Two MC elements $x,y\in\g$ are said to be \emph{gauge equivalent} if $x=g\cdot y$ for some $g\in\g^0$. We use $\sim$ to denote the corresponding equivalence relation.

\begin{defi}
Given a pro-nilpotent dg Lie algebra $\g$ we define its \emph{MC moduli set} $\MCmod(\g)$ as the set of equivalence classes
$\MC(\g)/{\sim}$ under gauge equivalence.
\end{defi}

If $A$ is a commutative dg algebra, we will write $\MCmod(\g,A)$ for $\MCmod(\g\otimes A)$.

Let us now discuss the important notion of \emph{homotopy} of MC elements. First, let $k[t,dt]$ be the graded commutative $k$-algebra generated by one polynomial generator $t$ in degree $0$ and one exterior generator in degree $1$. The differential is defined by the rule $d(t)=dt$ and extended to the whole $k[t,dt]$ by the Leibniz rule. Note that there are two maps $k[t,dt]\to k$ given by setting $t=0$ or $t=1$. Note that $k[t,dt]$ is a path object for $k$ in the model category $\CDGA$ of commutative dg algebras. Note also that for any dg Lie algebra $\g$ the tensor product $\g\otimes k[t,dt]=:\g[t,dt]$ is a dg Lie algebra and evaluations at 0 and 1 determine two dg Lie algebra maps $\g[t,dt]\to\g$.
\begin{defi}
Let $\g$ be a nilpotent dg Lie algebra.	Two MC elements $x,y\in\g$ are called \emph{Sullivan homotopic} if there exists $z\in\MC(\g[t,dt])$ such that $z|_{t=0}=x$ and $z|_{t=1}=y$.
\end{defi}
An important theorem due to Schlessinger and Stasheff \cite{SStasheff} shows that homotopy and gauge equivalence are equivalent notions for nilpotent dg Lie algebras.

\begin{thm}\label{th:SS}
Let $\g$ be a nilpotent dg Lie algebra. Then two MC elements $x,y\in\g$ are Sullivan homotopic if and only if they are gauge equivalent. In particular, the relation of homotopy on $\MC(\g)$ is an equivalence relation.
\end{thm}

\begin{proof}
See, e.g.~\cite{Chlazarev}.
\end{proof}

\begin{rem}
The construction $\g[t, dt]:=\g\otimes k[t,dt]$ used in the definition of Sullivan homotopy makes sense for any dg Lie algebra $\g$. For a general  dg Lie algebra $\g$ one does not expect to get a reasonable definition of an equivalence of MC elements in $\g$ using this construction. Suppose that $\g$ is pro-nilpotent, in that case we define $\g[t,dt]:=\varprojlim_n(\g/\g^{[n]}[t,dt])$ and modify the notion of homotopy of MC elements accordingly. It is easy to see that Schlessinger-Stasheff theorem \ref{th:SS} remains valid in this context. Moreover, Theorem \ref{th:SS} has a natural interpretation in terms of model categories: it says, roughly speaking, that the notions of left and right homotopy for nilpotent dg Lie algebras agree (see \cite{Aiguan} for a precise statement and its generalizations).
\end{rem}

\subsection{MC moduli in dg algebras}
We will now outline a parallel treatment of MC moduli for associative augmented dg algebras. It will be convenient for us to work with \emph{non-unital} dg algebras, i.e. dg algebras not necessarily possessing a unit. It is well-known that the categories of non-unital dg algebras and of \emph{augmented} dg algebras are equivalent: given a non-unital dg-algebra $\g$ one can adjoin a unit forming an augmented dg algebra $\g_e:=\g\oplus k\cdot 1$, and conversely, any augmented dg algebra gives rise to a non-unital dg algebra, its augmentation ideal.
\begin{defi}\label{defi:MCass}
Let $\g$ is a non-unital dg algebra over a field $k$ of arbitrary characteristic.
An element $x \in \g$ is called an \emph{MC element} if it satisfies $d(x) + x^2 = 0$.
The set of all MC elements in $\g$ will be denoted by $\MC(\g)$.
\end{defi}

Assume from now on that the non-unital dg algebra $\g$ is \emph{pro-nilpotent}. In other words, we have $\g=\varprojlim_n\g/\g^{[n]}$; here $\g^{[n]}$ is the dg ideal of $\g$ generated by products of at least $n$ elements.  Clearly the elements of $\g_e$ of the form $1+i$ where $ i\in \g$, are invertible, and therefore form a group $G$ that we will call the \emph{gauge group} associated to $\g$.

\begin{prop}
Let $g \in G$ and $x \in \MC(\g)$. Then the formula $g\cdot x:=gxg^{-1}-d(g)g^{-1}$ determines an action of $G$ on $\MC(\g)$.
\end{prop}

This action is well-defined by a similar argument as for Proposition~\ref{prop:gaugeaction}; this time we should make use of the associative algebra $\tilde{\g}$ having underlying space $\g\otimes k[d]$ where $d$ is a degree one element with $d^2=0$ (so that $k[d]$ is the exterior algebra on $d$ which can be viewed as the universal enveloping algebra of the abelian Lie algebra $k\cdot d$). The product in $\tilde{\g}$ is determined by requiring that $\g$ and $k[d]$ are subalgebras in $\tilde{\g}$ and there is a commutation relation $[d,a]=da-(-1)^{|a|}ad=d(a)$ for $a$ being a homogeneous element in $\g$ of degree $|a|$.
As before, we say that two MC elements $x,y\in \g$ are \emph{gauge equivalent} if $x=g\cdot y$ for some $g\in G$ and let $\sim$ denote the corresponding equivalence relation.

\begin{defi}
The \emph{MC moduli set} $\MCmod(\g)$ is the set of equivalence classes
$\MC(\g)/{\sim}$ under gauge equivalence.

As before, if $A$ is another dg algebra then we write $\MC(\g,A)$ for $\MC(\g\otimes A)$, and write $\MCmod(\g,A)$ for $\MCmod(\g\otimes A)$.
\end{defi}
The notion of \emph{homotopy} between two MC elements in an augmented dg algebra $\g$ can be treated in the same way as for dg Lie algebras, with an appropriate analogue of the Schlessinger-Stasheff, see \cite[Theorem 4.4]{CHL} where this approach is carried out in the smooth context. We will now describe a simple alternative way, that has the added advantage of not requiring that $k$ has characteristic zero.

Consider the dg algebra $\Int$ spanned by two vectors $a, b$ in degree $0$ and one vector $c$ in degree $1$. The differential is given by
\[d(a)=c,\quad d(b)=-c,\quad d(c)=0\] and the algebra structure is specified by
\[a^2=a, \quad b^2=b,
\quad ca=c,
\quad bc=c,
\quad ab=ba=c^2=0,\]
with unit element $1=a+b$. This is the cochain algebra on the standard cellular interval with two 0-cells corresponding to the endpoints and one 1-cell. The dg algebra $\g\otimes\Int$ is a path object for non-unital dg algebra $\g$. There are two `evaluation' maps $p_1,p_2:\Int\to k$ so that $p_1(c)=p_2(c)=0; p_1(a)=1, p_1(b)=1$, and these induce the corresponding evaluation maps $\g\otimes\Int\to \g$ required in the definition of the path object.
\begin{defi}\label{def:homass}
Let $\g$ be a non-unital dg algebra. Then two MC elements $x,y\in\g$ are \emph{homotopic}
if there exists $z\in\MC(\g\otimes\Int)$ such that $(1\otimes p_1)(z)=x$ and $(1\otimes p_2)(z)=y$.
\end{defi}
We have the following analogue of the Schlessinger-Stasheff theorem.
\begin{thm}\label{thm:SSass}
Let $\g$ be as in Definition \ref{def:homass}. Then two MC elements in $\g$ are homotopic if and only if they are gauge equivalent. In particular, the relation of homotopy on $\MC(\g)$ is an equivalence relation.	
\end{thm}
\begin{proof}
	Any element $z\in \I\otimes \Int \cong \Int\otimes \I$ may be written uniquely as $z = a\otimes z_1 + b\otimes z_2 + c\otimes h$ with $z_1,z_2, h\in\I$.
	The MC equation for $z$ is equivalent to  $z_1$ and $z_2$ being MC elements such that $d(h)=(1+h)z_1-z_2(1+h)$ inside $\g_e$. Since $1+h$ is invertible then the latter equation could be rewritten as $z_2=(1+h)z_1(1+h)^{-1}-d(1+h)(1+h)^{-1}$ so $z_1$ and $z_2$ are gauge equivalent.
\end{proof}		
	
\section{Koszul duality}

We will need a certain amount of theory of topological vector spaces, although we will be dealing with one of the simplest possible type of topological vector space -- pseudocompact spaces.
\begin{defi}
A \emph{pseudocompact vector space} is a topological vector space that is complete and whose fundamental system of neighbourhoods of zero is formed by subspaces of finite codimension. Morphisms of pseudocompact vector spaces are assumed to be continuous.
A \emph{graded pseudocompact vector space} is a graded object in the category of pseudocompact vector spaces, i.e.~a sequence $V^i$, $i\in\mathbb{Z}$ where each $V^i$ is a pseudocompact vector space with morphisms defined component-wise. Finally, a \emph{dg pseudocompact vector space} is a graded pseudocompact vector space $V^i$, $i\in\mathbb{Z}$ with a continuous differential.
\end{defi}

The categories of vector spaces (over $k$) and pseudocompact vector spaces will be denoted by $\Vect$ and $\pcVect$ respectively.
The categories of dg vector spaces and dg pseudocompact vector spaces will be denoted by $\DGVect$ and $\pcDGVect$ respectively.
For a dg (possibly pseudocompact) vector space $V$, its suspension $\Sigma V$ is the graded vector space $(\Sigma V)^i = V^{i+1}$.

\begin{prop}\label{prop:equiv}
The category $\Vect$ is anti-equivalent to $\pcVect$, and the category $\pcDGVect$ is anti-equivalent to $\DGVect$.
\end{prop}
\begin{proof}
Given a vector space $V$, its $k$-linear dual $V^*$ is pseudocompact. Indeed, denoting by $\{V_\alpha\}$ the collection of finite-dimensional subspaces of $V$, we have $V=\varinjlim_\alpha V_\alpha$ and therefore $V^*=\varprojlim V^*_\alpha$. So, $V^*$ is complete with respect to the kernels of maps into finite-dimensional spaces. The functor backwards associates to a pseudocompact vector space $V$ its \emph{continuous} linear dual $V^*$. It is straightforward to see that this gives the desired anti-equivalence. The dg case is similar.
\end{proof}	
The above proof shows that every (dg) pseudocompact vector space $V$ is a projective limit of its finite dimensional (dg) quotients $V_\alpha:V\cong\varprojlim_\alpha V_\alpha$. Conversely, a projective system of finite-dimensional dg vector spaces determines a dg pseudocompact vector space. Given two dg pseudocompact vector spaces $V\cong\varprojlim_\alpha V_\alpha$ and $U\cong\varprojlim_\beta U_\beta$ the dg (not pseudocompact in general) space of morphisms $V\to U$ is $\Hom(V,U)\cong \varprojlim_\beta\varinjlim_\alpha(V_\alpha, U_\beta)$.

Recall that the category $\DGVect$ has a symmetric monoidal structure given by the usual tensor product. Similarly for two
dg pseudocompact vector spaces $V=\varprojlim_\alpha V_\alpha$ and $U=\varprojlim_\beta U_\beta$ their completed tensor product is defined as $V \mathbin{\hat\otimes} U:=\varprojlim_{\alpha,\beta}V_\alpha\otimes U_\beta$. We will omit the hat over the symbol of the tensor product as it will always be understood. With this definition the anti-equivalence of Proposition \ref{prop:equiv} is that of symmetric monoidal categories, i.e.~there are natural isomorphisms $(V\otimes U)^*\cong V^*\otimes U^*$ where $U$ and $V$ are either dg vector spaces or pseudocompact vector spaces.

\subsection{DG coalgebras and pseudocompact dg algebras} Just as a (commutative) dg algebra can be defined succinctly as a (commutative) monoid in the symmetric monoidal category $\DGVect$, a (cocommutative) dg coalgebra is defined as a (cocommutative) comonoid in $\DGVect$. Using the monoidal anti-equivalence of Proposition \ref{prop:equiv} we see that the category of (cocommutative) dg coalgebras is anti-equivalent to category of (commutative) pseudocompact dg algebras, that is, (commutative) monoids in $\pcDGVect$. We will denote the latter category by $\pcDGA$ and $\pcCDGA$ in the commutative case.

The following result is a dg version of the so-called \emph{fundamental theorem on coalgebras}.
\begin{thm}\label{th:fundamental}
Any (cocommutative) dg coalgebra is a union of its finite-dimensional dg subcoalgebras.
\end{thm}
\begin{proof}
The non-dg version of the theorem is well-known, cf.~for example \cite[Theorem 2.2.1]{Sweedler}. The dg version is an easy consequence since any (possibly non-differential) subcoalgebra $A$ of a dg coalgebra $C$ is contained in the dg subcoalgebra $A+d(A)$ which is clearly finite-dimensional.
\end{proof}
\begin{cor}\label{cor:fund}
Any (commutative) pseudocompact dg algebra is the projective limit of its finite-dimensional quotients.
\end{cor}
\begin{proof}
Given a (commutative) pseudocompact dg algebra $A$, its $k$-linear dual $A^*$ is a (cocommutative) dg coalgebra. Then the desired statement is equivalent to saying that $A^*$ is an inductive limit (i.e.~a union) of its finite-dimensional dg subcoalgebras which is Theorem \ref{th:fundamental}.	
\end{proof}
\begin{rem}
Theorem \ref{th:fundamental} (and hence, Corollary \ref{cor:fund}) uses the associativity condition in an essential way and does not hold for other algebraic structures (e.g.~Lie coalgebras), see \cite[Section 2.4]{Positselski'} for an example of a Lie coalgebra possessing no proper Lie subcoalgebras at all.
\end{rem}
We will consider \emph{coaugmented} dg coalgebras, i.e.~dg coalgebras $C$ supplied with a dg coalgebra map $k\to C$. In this case the quotient $C/k$ is a dg coalgebra without a counit. Given a dg coalgebra $C$ we denote by $\Delta=\Delta^1:A\to A\otimes A$ its comultiplication and by $\Delta^n:A\to A^{\otimes n}$ its $n$th iteration. If for a coaugmented dg coalgebra $C$ we have $C/k=\bigcup_{n=1}^\infty\Ker(\Delta^n)$ then $C$ is called \emph{conilpotent}. It is easy to see that $C$ is conilpotent if an only if its dual pseudocompact dg algebra $C^*$ is augmented and for its augmentation ideal $I$ it holds that $C^*\cong \varprojlim_nC^*/I^n$. In other words, $C^*$ is a local complete augmented dg algebra with the maximal dg ideal $I$. Note also that if an augmented pseudocompact dg algebra is local (i.e.~its augmentation ideal $I$ is a unique dg maximal ideal) then it is automatically $I$-adically complete since its every ideal with finite-dimensional quotient must contain some power of $I$ and so its every finite-dimensional quotient factors through a power of $I$. All told, we have the following result.
\begin{prop}
The category of (cocommutative) conilpotent dg coalgebras is anti-equivalent to the category of local augmented (commutative) pseudocompact dg algebras.
We will denote that latter category by $\pcDGAloc$ and $\pcCDGAloc$ in the commutative case.
\qed
\end{prop}

\subsection{Quillen equivalence between $\DGLA$ and $\pcCDGAlocop$}

We now explain a Quillen equivalence between $\DGLA$ and $\pcCDGAlocop$ due to Hinich \cite{Hinich}, also called \emph{Koszul duality}, which is at the heart of the modern approach to deformation theory. We assume that the ground field $k$ has characteristic zero. A similar approach works for algebras and (suitably defined) local pseudocompact algebras over a pair of Koszul dual operads; we will not pursue this in full generality but consider, later on, an associative analogue of this story.

Any local augmented pseudocompact commutative dg algebra $A$ with augmentation ideal $I(A)$ determines a dg Lie algebra as follows.
\begin{defi}\label{def:har}
Let $A\in \pcCDGAloc$ and set $\Harr(A)$ to be the dg Lie algebra whose underlying space is the free Lie algebra on $\Sigma^{-1} I(A)^*$ and the differential $d$ is defined as $d=d_{I}+d_{II}$; here $d_I$ is induced by the internal differential on $I(A)$ and $d_{II}$ is determined by its restriction onto $\Sigma^{-1} I(A)^*$ which is in turn induced by the product map $I(A)\otimes I(A)\to I(A)$.
\end{defi}
\begin{rem}
Note that since $I(A)$ is pseudocompact, its dual $I(A)^*$ is discrete and thus, the dg Lie algebra $\Harr(A)$ is a conventional dg Lie algebra (with no topology). The construction $\Harr(A)$ is the continuous version of the Harrison complex associated with a commutative dg algebra.
\end{rem}

Similarly, any dg Lie algebra determines a local pseudocompact commutative dg algebra as follows.
\begin{defi}\label{def:ce}For a dg Lie algebra $\g$ set $\CE(\g)=\hat{S}\Sigma^{-1} \g^*$, the completed symmetric algebra on $\Sigma^{-1}\g^*$.
The differential $d$ on $\CE(\g)$ is defined as
$d=d_{I}+d_{II}$; here $d_I$ is induced by the internal differential on $\g$ and
$d_{II}$ is determined by its restriction onto $\Sigma^{-1} \g^*$ which is in turn induced by the bracket map $\g\otimes \g\to \g$.	
\end{defi}
The following result holds.
\begin{prop}
The functors $\Harr:\pcCDGAlocop\rightleftarrows\DGLA:\CE$ form an adjoint pair.
\end{prop}

\begin{proof}
We only need to notice that for $A\in\pcCDGAloc$ and $\g\in\DGLA$ there are natural isomorphisms
\[
\Hom_{\DGLA}(\Harr(A),\g) \cong \MC(\g\otimes A)\cong\Hom_{\pcCDGAloc}(\CE(\g),A). \qedhere
\]
\end{proof}	
The category $\pcCDGAloc$ has the structure of a model category.
\begin{defi}
A morphism $f:A\to B$ in $\pcCDGAloc$ is called
\begin{enumerate}
\item
a \emph{weak equivalence} if
$\Harr(f): \Harr(B)\to \Harr(A)$
is a quasi-isomorphism of dg Lie algebras;
\item
a \emph{fibration} if $f$ is surjective;
\item
a \emph{cofibration} if $f$ has the LLP with respect to all acyclic fibrations.
\end{enumerate}
\end{defi}

\begin{thm}\label{th:hinich}
The category $\pcCDGAloc$ together with the classes of fibrations, cofibrations and weak equivalences is a model category. Moreover, the adjoint pair of functors $(\Harr,\CE)$ is a Quillen equivalence between $\pcCDGAlocop$ and $\DGLA$.
\end{thm}

\begin{proof}
See \cite{Hinich}.
\end{proof}

\begin{rem}\label{rem:rightpropercdga}
By definition, all objects in the $\pcCDGAloc$ are fibrant, so by Proposition \ref{prop:rightproper} it is right proper.
\end{rem}

The notion of the MC moduli set has a natural interpretation in terms of model structures on $\DGLA$ and $\pcCDGAloc$.

\begin{thm}\label{th:homotopy}
Let $\g$ be a dg Lie algebra and $A$ be a local pseudocompact dg algebra. Then there are the following isomorphisms, natural in both variables:
\[
 [\Harr(A),\g]\cong\MCmod(\g,A)\cong[\CE(\g),A].
\]	
\end{thm}

\begin{proof}
The bijection $[\Harr(A),\g]\cong[\CE(\g),A]$ follows from the adjunction $(\Harr,\CE)$ on the level of homotopy categories. Since $\Harr(A)$ is a cofibrant dg Lie algebra, $[\Harr(A),\g]$ can be identified with Sullivan homotopy classes of maps $\Harr(A)\to \g$ (choosing $\g[t,dt]$ as a path object for $\g$) and the latter set can be identified, by Theorem \ref{th:SS} with $\MC(\g\otimes A)$ modulo gauge equivalence, i.e.~with $\MCmod(\g,A)$. Note that $\g\otimes A$ may not be nilpotent, so we need a pro-nilpotent version of Theorem \ref{th:SS}.
\end{proof}		
\begin{rem}
A weak equivalence in $\pcCDGAloc$ is \emph{not} the same as a quasi-isomorphism. Indeed, let $\g$ be the ordinary Lie algebra $\sll_2(k)$. It is well-known that the Chevalley-Eilenberg cohomology of $\sll_2(k)$ is $\Lambda(x)$, the exterior algebra on one generator $x$ in degree $3$ and it follows that $\CE(\g)$ is formal, i.e.~quasi-isomorphic to its own cohomology. However, $\CE(\g)$ is not weakly equivalent to $\Lambda(x)$, for if it were, then the dg Lie algebra $\Harr(\CE(\g))$ would be on the one hand, quasi-isomorphic to $\g$ by Theorem \ref{th:hinich}, and on the other, to $\Harr(\Lambda(x))$. But $\Harr(\Lambda(x))$ is isomorphic to the abelian Lie algebra with one basis vector in degree $2$ and it is, of course, not quasi-isomorphic to $\g=\sll_2(k)$. In fact, a weak equivalence in $\pcCDGAloc$ is that of a \emph{filtered} quasi-isomorphism and it is finer that a quasi-isomorphism: every weak equivalence of local commutative pseudocompact dg algebras is a quasi-isomorphism but not vice-versa.
\end{rem}	
\begin{prop}\label{prop:compact}
The category $\pcCDGAlocop$ is compactly generated.
\end{prop}
\begin{proof}
Let us denote by $X_n, n\in \mathbb{Z}$ the commutative algebra $k \oplus \Sigma^n k$ where $\Sigma^n k$ has zero multiplication. We claim that the set $\{X_n, n\in \mathbb{Z}\}$ forms a set of compact generators for $\pcCDGA^{\op}$. To see that note that under the Quillen equivalence of Theorem \ref{th:hinich}, the algebra $X_n$ corresponds to the free Lie algebra on one generator in degree $n-1$. These free Lie algebras clearly form a set of compact generators for dg Lie algebras so the conclusion follows.
\end{proof}	

\subsection{Quillen equivalence between $\DGAa$ and $\pcDGAlocop$}

We now explain an associative analogue of the picture of Koszul duality from the previous section, where $\pcCDGAlocop$ is replaced with $\pcDGAloc$ (i.e.~the commutativity is dropped) and $\DGLA$ is replaced with $\DGAa$, the category of \emph{augmented} dg algebras. cf.~\cite{Positselski}. Here we do not insist that the ground field $k$ has characteristic zero.


Any local augmented pseudocompact dg algebra $A$ with augmentation ideal $I(A)$ determines an augmented dg algebra as follows.

\begin{defi}\label{def:cobar}
For $A \in \pcDGAloc$ set $\Cobar(A) = T\Sigma^{-1} I(A)^*$, the uncompleted tensor algebra on the discrete vector space $\Sigma^{-1}I(A)^*$.
The differential $d$ on $\Cobar(A)$ is defined as $d=d_{I}+d_{II}$; here $d_I$ is induced by the internal differential on $I(A)$ and $d_{II}$ is determined by its restriction onto $\Sigma^{-1} I(A)^*$ which is in turn induced by the product map $I(A)\otimes I(A)\to I(A)$.
\end{defi}

Similarly, any augmented dg algebra $\g$ with augmentation ideal $\I(\g)$ determines a local pseudocompact dg algebra as follows.

\begin{defi}\label{def:bar}
For $\g \in \DGAa$ set $\Bar(\g)=\hat{T}\Sigma^{-1} \I(\g)^*$, the completed tensor algebra on $\Sigma^{-1}\I(\g)^*$.
The differential $d$ on $\Bar(\g)$ is defined as
$d=d_{I}+d_{II}$; here $d_{I}$ is induced by the internal differential on $\g$ and
$d_{II}$ is determined by its restriction onto $\Sigma^{-1} \I(\g)^*$ which is in turn induced by the product map $\I(\g)\otimes \I(\g)\to \I(\g)$.	
\end{defi}

\begin{rem}
The construction $\Bar(\g)$ is commonly referred to as the bar-construction of the dg algebra $\g$; its cohomology computes $\operatorname{Ext}_{\g}(k,k)$. Similarly, $\Cobar(A)$ is the cobar-construction of the pseudocompact dg algebra $A$ (or its dual dg coalgebra).
\end{rem}

The following result holds.
\begin{prop}\label{prop:assocadj}
The functors $\Cobar: \pcDGAlocop \leftrightarrows \DGAa :\Bar$ form an adjoint pair.
\end{prop}

\begin{proof}
We only need to notice that for $A\in\pcDGAloc$ and $\g\in\DGAa$ there are natural isomorphisms
\[
\Hom_{\DGAa}(\Cobar(A),\g) \cong \MC(\I(\g)\otimes A) \cong \Hom_{\pcDGAloc}(\Bar(\g),A). \qedhere
\]
\end{proof}	
The category $\pcDGAloc$ has the structure of a model category.
\begin{defi}
A morphism $f:A\to B$ in $\pcDGAloc$ is called
\begin{enumerate}
\item
a \emph{weak equivalence} if
$\Cobar(f): \Cobar(B)\to \Cobar(A)$
is a quasi-isomorphism of dg algebras;
\item
a \emph{fibration} if $f$ is surjective;
\item
a \emph{cofibration} if $f$ has the LLP with respect to all acyclic fibrations.
\end{enumerate}
\end{defi}

\begin{thm}\label{th:positselski}
The category $\pcDGAloc$ together with the classes of fibrations, cofibrations and weak equivalences is a model category. Moreover, the adjoint pair of functors $(\Cobar,\Bar)$ is a Quillen equivalence between $\pcDGAlocop$ and $\DGAa$.
\end{thm}

\begin{proof}
See \cite{Positselski}.
\end{proof}

\begin{rem}\label{rem:rightproperdga}
By definition, all objects in the $\pcDGAloc$ are fibrant, so by Proposition \ref{prop:rightproper} it is right proper.
\end{rem}

\begin{thm}\label{th:assochomotopy}
There are the following isomorphisms, natural in both variables:
\[
[\Cobar(A),\g]_{\DGAa}\cong\MCmod(\I(\g),A) \cong[\Bar(\g),A]_{\pcDGAloc}.
\]	
\end{thm}

\begin{proof}
The proof is the same as that of Theorem~\ref{th:homotopy} with $\Harr(A)$ and $\CE(\g)$ replaced by $\Cobar(A)$ and $\Bar(\g)$ respectively. The only difference is that we choose the smaller path object $\I(\g) \otimes \Int$ for $\I(\g) $ and apply Theorem~\ref{thm:SSass} to identify homotopy classes of maps $\Cobar(A) \to \g$ with $\MCmod(\I(\g),A)$.
\end{proof}		
\begin{rem}
	A weak equivalence in $\pcDGAloc$ is \emph{not} the same as a quasi-isomorphism. Indeed, let $\g$ be ordinary associative algebra $k\times k$, the product of two copies of $k$. Then $\Bar(\g)$ is easily seen to be the dual to the bar-resolution of the algebra $k$, in particular it is quasi-isomorphic to $k$. If it were weakly equivalent to $k$ in $\pcDGAloc$ then $\Cobar(\Bar(\g))$ would be, on the one hand, quasi-isomorphic to $\g\cong k\times k$ and, on the other, to $\Cobar(k)\cong k$ giving a contradiction.  In fact, a weak equivalence in $\pcDGAloc$ is that of a filtered quasi-isomorphism and it is finer that a quasi-isomorphism: every weak equivalence of local
	 pseudocompact dg algebras is a quasi-isomorphism but not vice-versa.
\end{rem}
\begin{prop}\label{prop:assoccompact}
The category $\pcDGAlocop$ is compactly generated.
\end{prop}
\begin{proof}
	The argument is the same as in Proposition \ref{prop:compact}, using Theorem \ref{th:positselski} in place of Theorem \ref{th:hinich}.
\end{proof}	
\subsection{Relationship between two types of Koszul duality} We will now discuss how the associative Koszul duality is related to the Lie-commutative one.

Given a dg Lie algebra $\g$, its universal enveloping algebra $U\g$ is a dg algebra; this determines a functor $\DGLA\to\DGAa$ that is left adjoint to the functor $\Lie$ taking an associative augmented dg algebra to the commutator dg Lie algebra of its augmentation ideal.  Similarly the forgetful functor $\As:\pcCDGAloc\to \pcDGAloc$ is right adjoint to the \emph{abelianization} functor $\Ab:\pcDGAloc\to \pcCDGAloc $, associating to an associative pseudocompact dg algebra $\g$ its quotient by the ideal topologically generated by (graded) commutators in $\g$. It is clear that both are in fact Quillen adjunctions.

\begin{prop}\label{prop:comass}
	The following diagrams of model categories and Quillen functors between them is commutative in the sense that there is a functor isomorphism $U\circ\Harr\cong \Cobar\circ \As$ and  $\CE\circ\Lie\cong \Ab\circ\Bar$.
	\[
\xymatrix{
\DGAa&\DGLA\ar_U[l]&\DGAa\ar_{\Bar}[d]\ar^{\Lie}[r]&\DGLA\ar^{\CE}[d]\\
\pcDGAlocop\ar^{\Cobar}[u]&
\pcCDGAlocop\ar_{\Harr}[u]\ar_{\As}[l]&\pcDGAlocop\ar^{\Ab}[r]&\pcCDGAlocop
}	
	\]
	\end{prop}
\begin{proof}
	Straightforward unravelling of the definitions.
\end{proof}	

\section{Main theorems}

\subsection{MC elements and the deformation functor based on a dg Lie algebra}
Any dg Lie algebra $\g$ determines a deformation functor $\Def_{\g}:A\mapsto\Def_\g(A)=\MCmod(\g, A)$ where $A$ is a local pseudocompact commutative dg algebra. Thus, $\Def_\g$ is a set-valued functor on $\pcCDGAloc$  This (extended) deformation functor has the following homotopy invariance property.

\begin{thm}\label{th:invariance}Let $\g$ be a dg Lie algebra.
\begin{enumerate}
\item If $A\to B$ is a weak equivalence in $\pcCDGAloc$ then the induced map $\Def_{\g}(A)\to\Def_{\g}(B)$ is an isomorphism. Therefore $\Def_{\g}$ descends to a set-valued functor on $\ho(\pcCDGAloc)$ that will be denoted by the same symbol.
\item If $\g$ and $\g^\prime$ are two quasi-isomorphic dg Lie algebras, then the functors $\Def_{\g}$ and $\Def_{\g^\prime}$ are isomorphic.
\end{enumerate}
\end{thm}

\begin{proof}
This follows from Theorem \ref{th:homotopy}.
\end{proof}

\begin{thm}\label{th:representable}
The set-valued functor $\Def_{\g}$ on $\ho(\pcCDGAloc)$ is representable by the local pseudocompact commutative dg algebra $\CE(\g)$. Conversely, any functor on $\ho(\pcCDGAloc)$ that is homotopy representable by a local pseudocompact commutative dg algebra $A$ is isomorphic to the functor $\Def_{\Harr(A)}$.
\end{thm}

\begin{proof}
By Theorem \ref{th:homotopy} we have
$\Def_{\g}(A)=\MCmod(\g,A)\cong [\CE(\g),A]$, which means that $\Def_{\g}$ is representable by $\CE(\g)$. Conversely, given a functor $F$ on $\pcCDGAloc$ representable by a local pseudocompact dg algebra $A$ we have for $B\in \pcCDGAloc$:
\begin{align*}
F(B) &= [B,A]\\
&\cong [\CE(\Harr(A)),B]\\
&\cong \MCmod(\Harr(A), B)\\
&\cong \Def_{\Harr(A)}(B)
\end{align*}
as required.
\end{proof}

\subsection{Finding a dg Lie algebra associated with a deformation functor}
We will now formulate the necessary and sufficient conditions on a homotopy invariant functor on $\pcCDGA$ ensuring that it is representable (and thus, `controlled' by a dg Lie algebra).

\begin{thm}\label{th:main:cdga}
Let $F$ be a set-valued functor on $\pcCDGAloc$ such that:
\begin{enumerate}
\item $F$ is homotopy invariant: it takes weak equivalences in $\pcCDGAloc$ to bijections of sets.
\item $F$ is normalized: $F(k)$ is a one-element set.
\item $F$ takes arbitrary products in $\pcCDGAloc$ into products of sets.
\item For any diagram in $\pcCDGAloc$ of the form $B\rightarrow A\leftarrow C$ where  $A\leftarrow C$  is surjective, the natural map
$F(B\times_AC)\to F(B)\times_{F(A)}F(C)$ is surjective.
\end{enumerate}
Then $F$ is homotopy representable, i.e.  there exists $X\in\pcCDGAloc$ such that for any $Y\in\pcCDGAloc$ there is a natural isomorphism $F(Y)\cong [X,Y]$.
\end{thm}

\begin{proof}
This follows from Brown representability, Theorem \ref{thm:brownrep}, taking into account that the model category $\pcCDGAlocop$ is compactly generated, cf.~Proposition \ref{prop:compact}.
\end{proof}	

\begin{rem}
One can consider deformation functors with values in simplicial sets, rather than sets. This is the approach taken in \cite{Lurie,Pridham}. There is a version of the representability theorem in this setting.
\end{rem}

\subsection{Associative deformation theory}
Any augmented dg algebra $\g$ over a field $k$ of arbitrary characteristic determines a deformation functor  $\Def_{\g}:A\mapsto\Def_\g(A)=\MCmod(\g, A)$ where $A$ is a local pseudocompact associative dg algebra. Thus, $\Def_\g$ is a set-valued functor on $\pcDGAloc$.  This (extended) deformation functor has the following homotopy invariance property.

\begin{thm}\label{th:associnvariance}
Let $\g$ be an augmented dg algebra.
\begin{enumerate}
\item If $A\to B$ is a weak equivalence in $\pcDGAloc$ then the induced map $\Def_{\g}(A)\to\Def_{\g}(B)$ is an isomorphism. Therefore $\Def_{\g}$ descends to a set-valued functor on $\ho(\pcDGAloc)$ that will be denoted by the same symbol.
\item If $\g$ and $\g^\prime$ are two quasi-isomorphic dg algebras, then the functors $\Def_{\g}$ and $\Def_{\g^\prime}$ are isomorphic.
\end{enumerate}
\end{thm}

\begin{proof}
This follows from Theorem \ref{th:assochomotopy}.
\end{proof}

\begin{thm}\label{th:defass}
The set-valued functor $\Def_{\g}$ on $\ho(\pcDGAloc)$ is representable by the local pseudocompact dg algebra $\Bar(\g)$. Conversely, any functor on $\ho(\pcDGAloc)$ that is homotopy representable by a local pseudocompact dg algebra $A$ is isomorphic to the functor $\Def_{\Cobar(A)}$.
\end{thm}

\begin{proof}
The proof is the same as that of Theorem~\ref{th:representable}, applying Theorem~\ref{th:assochomotopy} instead of Theorem~\ref{th:homotopy}.
\end{proof}

\subsection{Finding a dg algebra associated with a deformation functor} We will now formulate the necessary and sufficient conditions on a homotopy invariant functor on $\pcDGAloc$ ensuring that it is representable (and thus, `controlled' by an augmented (or, equivalently, non-unital) dg algebra).

\begin{thm}
Let $F$ be a set-valued functor on $\pcDGAloc$ such that:
\begin{enumerate}
\item $F$ is homotopy invariant: it takes weak equivalences in $\pcDGAloc$ to bijections of sets;
\item $F$ is normalized: $F(k)$ is a one-element set.
\item $F$ takes arbitrary products in $\pcDGAloc$ into products of sets.
\item For any diagram in $\pcDGAloc$ of the form $B\to A\leftarrow C$ where  $A\leftarrow C$  is surjective, the natural map
$F(B\times_AC)\to F(B)\times_{F(A)}F(C)$ is surjective.
\end{enumerate}
Then $F$ is homotopy representable, i.e.  there exists $X\in\pcDGAloc$ such that for any $Y\in\pcCDGAloc$ there is a natural isomorphism $F(Y)\cong [X,Y]$.
\end{thm}

\begin{proof}
This follows from Brown representability, Theorem \ref{thm:brownrep}, taking into account that the model category $\pcDGAlocop$ is compactly generated, cf.~Proposition \ref{prop:compact}.
\end{proof}
\subsection{Comparing commutative and associative deformations}
Assume now that $k$ has characteristic zero. Any set-valued functor $F$ on $\pcDGAloc$ determines by restriction a functor on $\pcCDGAloc$ and so it makes sense to ask whether an associative deformation functor $\Def_\g$ for $\g\in\DGAa$ restricts to a deformation functor on $\pcCDGAloc$. The following results answer this question.
\begin{thm}
Let $\g$ be a dg algebra. Then	the deformation functor $\Def_\g$ on $\pcDGAloc$ restricts to
 the deformation functor $\Def_{\Lie(\g)}$ on $\pcCDGAloc$.
\end{thm}
\begin{proof}
	We know by Theorem \ref{th:defass} that $\Def_\g$ is represented by a dg algebra $\Bar(\g)$. Then for $\h\in\pcCDGAloc$ we have
	$\Def_\g(\h)=[\Bar(\g),\h]_{\DGAa}$ and so by Proposition \ref{prop:comass} and Theorem \ref{th:homotopy} we have:
	\begin{align*}
\Def_\g(\h)&\cong [\Ab(\Bar(\g)),\h]_{\pcCDGAloc}\\
	&\cong [\CE(\Lie(\g)),\h]_{\pcCDGAloc}\\&\cong \MCmod(\Lie(\g),\h)\\
	&\cong \Def_{\Lie(\g)}(\h)
	\end{align*}
	as claimed.
\end{proof}
\begin{rem}
As we saw, every deformation functor in characteristic zero is controlled by a dg Lie algebra. On the other hand, not every deformation functor is defined on the category $\pcDGAloc$ (which would imply that it is controlled by an \emph{associative} dg algebra), in the same way as not every Lie algebra comes from an associative algebra. An interesting example of an associative deformation theory is that of deformations of modules over an associative algebra.	Let $\g$ be an algebra and $M$ be a $\g$-module. Deformations of $M$ are controlled by the dg algebra $\REnd(M)$  the derived endomorphism
	algebra of $M$ viewed as a non-unital algebra; (it can be obtained as the ordinary endomorphism algebra of a $\g$-projective resolution of $M$). Considered as a functor on $\pcCDGAloc$, this deformation theory is controlled by $\Lie(\REnd(M))$, the commutator Lie algebra of $\REnd(M)$. More generally, deformations of $A_\infty$-modules over an $A_\infty$ algebra are controlled by a certain non-unital dg algebra, cf. \cite{GLST} regarding this example.
\end{rem}	

\noindent{\bf Acknowledgements.} The second author thanks J. Pridham for useful discussions. This research was partially supported by NSFC (11922110).

\setlength{\parindent}{0pt}
\end{document}